\documentclass[11pt,reqno]{amsart}
\usepackage{graphicx}
\usepackage{amsmath}
\usepackage{amsfonts}
\usepackage{amssymb}
\usepackage{amsthm}
\usepackage{hyperref}

\DeclareSymbolFont{rsfscript}{OMS}{rsfs}{m}{n}
\DeclareSymbolFontAlphabet{\mathrsfs}{rsfscript}

\newcommand{\Rg}{\mathrel{\mathrsfs{R}}}
\newcommand{\Lg}{\mathrel{\mathrsfs{L}}}
\newcommand{\Hg}{\mathrel{\mathrsfs{H}}}

\newtheorem{lemma}{Lemma}
\newtheorem*{theorem}{Theorem}

\begin{document}

\title[Periodic elements of the free semigroup on a biordered set]{Periodic elements of the free idempotent generated semigroup on a biordered set}

\author[Easdown, Sapir and Volkov]{D. Easdown \and M. V. Sapir \and M. V. Volkov}

\subjclass{20M05}
\renewcommand{\subjclassname}{\textup{2000} Mathematics Subject Classification}

\thanks{The second author was supported in part by the NSF grant DMS-0700811 and
by a BSF (USA-Israeli) grant. The third author acknowledges support
from the Russian Foundation for Basic Research, grant 06-01-00613.}

\address{First author's address: School of Mathematics and Statistics,
University of Sydney, NSW 2006, Australia.}
\email{\href{mailto:de@maths.usyd.edu.au}{de@maths.usyd.edu.au}}

\address{Second author's address: Department of Mathematics, 1326 Stevenson
Center, Vanderbilt University, Nashville, TN 37240, USA.}
\email{\href{mailto:m.sapir@vanderbilt.edu}{m.sapir@vanderbilt.edu}}

\address{Third author's address: Department of Mathematics and Mechanics,
Ural State University, Lenina~51, 620083 Ekaterinburg, Russia.}
\email{\href{mailto:Mikhail.Volkov@usu.ru}{Mikhail.Volkov@usu.ru}}

\begin{abstract}
We show that every periodic element of the free idempotent
generated semigroup on an arbitrary biordered set belongs
to a subgroup of the semigroup.
\end{abstract}

\maketitle

The \emph{biordered set of a semigroup} $S$ is the set of idempotents of $S$
considered as a partial groupoid with respect to the restriction of the multiplication
of $S$ to those pairs $(e,f)$ of idempotents such that $ef=e$, $ef=f$, $fe=f$ or $fe=f$.
Nambooripad~\cite{Nam79} who has initiated an axiomatic approach to biordered sets has
defined an \emph{abstract biordered set} as a partial groupoid satisfying certain second
order axioms. The first author~\cite{Eas85} has confirmed the adequacy of Nambooripad's
axiomatization by showing that each abstract biordered set is in fact the biordered set
of a suitable semigroup. Namely, if $\langle E,\circ\rangle$ is an abstract biordered set,
denote by $IG(E)$ the semigroup with presentation
$$IG(E) = \{E \mid ef = e\circ f\  \text{whenever $e\circ f$ is defined in $E$}\}.$$
The semigroup $IG(E)$ is called the \emph{free idempotent generated semigroup on $E$}.
In~\cite{Eas85} it has been shown that the biordered set of $IG(E)$ coincides with
the initial biordered set $\langle E,\circ\rangle$ (see Lemma~\ref{Easdown} below
for a precise formulation of this result).

The structure of the free idempotent generated semigroup on a biordered set is
not yet well understood. It was conjectured that subgroups of such a semigroup
should be free. Though confirmed for some partial cases (see \cite{McE02,NP80,Pas77,Pas80}),
this conjecture has been recently disproved by Brittenham, Margolis, and Meakin~\cite{BMM1}
who have found a biordered set $\langle E,\circ\rangle$ such that the semigroup $IG(E)$
has the free abelian group of rank~2 among its subgroups. Moreover, in the subsequent paper
\cite{BMM2} the same authors have proved that if $F$ is any field, and $E_3(F)$ is
the biordered set of the monoid of all $3\times 3$ matrices over $F$, then the free
idempotent generated semigroup over $E_3(F)$ has a subgroup isomorphic to the
multiplicative group of $F$. In particular, letting $F$ be the field of complex numbers,
one concludes that the free idempotent generated semigroup on a biordered set can contain
group elements of any finite order.

Recall that an element $a$ of a semigroup $S$ is said to be \emph{periodic} if $a$
generates a finite subsemigroup in $S$; in other words, if
\begin{equation}
\label{index and period}
a^h=a^{h+d}
\end{equation}
for some positive integers $h$ and $d$. Given $a$, the least $h$ and $d$ verifying
the equality~\eqref{index and period} are called respectively the \emph{index} and
the \emph{period} of $a$. The aforementioned discovery by Brittenham, Margolis, and
Meakin~\cite{BMM2} shows that there is no restriction to periods of periodic elements
in the free idempotent generated semigroup on a biordered set. The main result of the
present note demonstrates that, in contrast, indices of periodic elements in such
a semigroup are severely restricted, namely, they must be equal to 1. In other
words, we aim to show that every periodic element of $IG(E)$ must belong to
a subgroup of $IG(E)$.

We assume the reader's familiarity with Green's relations $\mathrsfs{L}$,
$\mathrsfs{R}$, $\mathrsfs{H}$ and their basic properties that can be found
in the early chapters of any general semigroup theory text. The following
property is also elementary but perhaps less known.
\begin{lemma}
Let $S$ be a semigroup, $a,e\in S$, $e^2=e$, $p,q$ positive integers where $p\le q$.
Then $a^p\Rg a^q=e$ implies $a^p\Hg e$.
\end{lemma}

\begin{proof}
Clearly, $e=a^{q-p}a^p\in S^1a^p$. Since $a^p=eb$ for some $b\in S^1$, we have
$$a^pe=a^{p+q}=ea^p=e(eb)=eb=a^p.$$
Thus, $a^p\in S^1e$, whence $a^p\Lg e$ and $a^p\Hg e$.
\end{proof}

We fix an arbitrary biordered set $\langle E,\circ\rangle$. Now let $E^+$ be the free
semigroup on $E$ and $\varphi:E^+\to IF(E)$ the onto morphism extending the identity
map on~$E$.

\begin{lemma}
\label{Easdown}
If $w\in E^+$ and $w\varphi$ is idempotent, then $w\varphi=e\varphi$ for some $e\in E$.
\end{lemma}
\begin{proof}
This is the main result of \cite{Eas85}.
\end{proof}

As usual, $E^*$ stands for $E^+$ with the empty word 1 adjoined.
\begin{lemma}
If $w\in E^+$ and $w\varphi=e\varphi$ for some $e\in E$, then there exist
$u,v\in E^*$ and $f\in E$ such that $w=ufv$ and $(uf)\varphi\Lg f\varphi\Rg (fv)\varphi$.
\end{lemma}

\begin{proof}
Let $\sigma=\ker\varphi$. Clearly, every two $\sigma$-related words in $E^+$ can be
connected by a sequence of elementary $\sigma$-transitions of the form $xe_1e_2y\to xe_3y$
or $xe_3y\to xe_1e_2y$ where $x,y\in E^*$, $e_1,e_2,e_3\in E$ and $e_1\circ e_2=e_3$ in
the biordered set $\langle E,\circ\rangle$. We induct on the minimum length $n$ of such
a sequence from $w$ to $e$. If $n=0$, that is $w=e$, the claim is obvious since we can
set $u=v=1$ and $f=e$. Suppose $n>0$ and let $w\to w'$ be the first $\sigma$-transition
in a sequence of minimum length connecting $w$ and $e$. By the induction assumption,
$w'=u'f'v'$ for some $u',v'\in E^*$ and $f'\in E$ such that $(u'f')\varphi\Lg f'\varphi\Rg
(f'v')\varphi$. On the other hand, for some $x,y\in E^*$, $e_1,e_2,e_3\in E$, we have
the decompositions $w=xe_1e_2y$, $w'=xe_3y$ (the contraction case) or $w=xe_3y$,
$w'=xe_1e_2y$ (the expansion case).

Consider the contraction case. We have $w'=xe_3y=u'f'v'$. First suppose that the
distinguished occurrence of $f'$ happens within the word $x$, that is $x=u'f'x'$,
$v'=x'e_3y$ for some $x'\in E^*$:
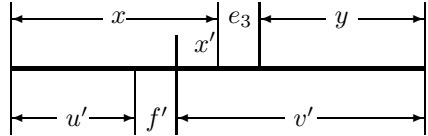
\begin{center}
\unitlength=1.1mm
\begin{picture}(50,20)
\linethickness{1.2pt}
\put(0,10){\line(1,0){50}}
\linethickness{0.4pt}
\put(0,2){\line(0,1){16}}
\put(25,10){\line(0,1){8}}
\put(30,10){\line(0,1){8}}
\put(15,2){\line(0,1){8}}
\put(20,2){\line(0,1){12}}
\put(50,2){\line(0,1){16}}
\put(5.5,4){\vector(-1,0){5.4}}
\put(9.5,4){\vector(1,0){5.4}}
\put(33,4){\vector(-1,0){12.9}}
\put(37,4){\vector(1,0){12.9}}
\put(14,16){\vector(1,0){10.9}}
\put(11,16){\vector(-1,0){10.9}}
\put(42,16){\vector(1,0){7.9}}
\put(38,16){\vector(-1,0){7.9}}
\small
\put(21.8,12){$x'$}
\put(15.9,3){$f'$}
\put(6.3,3){$u'$}
\put(33.8,3){$v'$}
\put(11.8,15.5){$x$}
\put(25.9,15.5){$e_3$}
\put(38.8,15.5){$y$}
\end{picture}
\end{center}
Then the word $w=xe_1e_2y$ also decomposes as $w=ufv$ where $u=u'$, $f=f'$,
and $v=x'e_1e_2y$ so that
$$(uf)\varphi = (u'f')\varphi\Lg f'\varphi = f\varphi$$
and
$$f\varphi = f'\varphi\Rg(f'v')\varphi = (f'x'e_3y)\varphi
               = (fx'e_1e_2y)\varphi = (fv)\varphi.$$
Thus,
$$(uf)\varphi\Lg f\varphi\Rg (fv)\varphi,$$
as required.

The situation when the distinguished occurrence of $f'$ happens within the word $y$
is handled in a symmetric way.

Now suppose that $x=u'$, $y=v'$ and $e_3=f'$. Then $f'=e_1\circ e_2$ in the biordered
set $\langle E,\circ\rangle$. By the definition of a biordered set, the product
$e_1\circ e_2$ is defined if and only if either 1)~$e_1\circ e_2=e_1$, or
2)~$e_1\circ e_2=e_2$, or 3)~$e_2\circ e_1=e_1$, or 4)~$e_2\circ e_1=e_2$.
In Cases~1 and~3 set $u=u'=x$ and $v=e_2y=e_2v'$. Then $w=ue_1v$. Since
$(u'f')\varphi\Lg f'\varphi\Rg (f'v')\varphi$ and $f'\varphi=(e_1e_2)\varphi$,
we have
$$(ue_1e_2)\varphi\Lg (e_1e_2)\varphi\Rg (e_1v)\varphi.$$
Under the condition of each of the cases under consideration, $(e_1e_2e_1)\varphi=e_1\varphi$
whence $e_1\varphi\Rg (e_1e_2)\varphi$. Multiplying the relation $(ue_1e_2)\varphi\Lg (e_1e_2)\varphi$
through on the right by $e_1\varphi$, we get $(ue_1)\varphi\Lg e_1\varphi$. Thus,
$$(ue_1)\varphi\Lg e_1\varphi\Rg (e_1v)\varphi,$$
as required. Cases~2 and~4 are dual.

Now consider the expansion case. We have $w'=xe_1e_2y=u'f'v'$. The situations when the distinguished occurrence
of $f'$ happens within $x$ or $y$ are completely similar to the analogous situations
in the contraction case. Suppose that $x=u'$, $e_1=f'$ and $e_2y=v'$. Then we set $u=u'=x$
and $v=y$, whence $w=ue_3v$. Since $(u'f')\varphi\Lg f'\varphi\Rg (f'v')\varphi$, we have
$$(ue_1)\varphi\Lg e_1\varphi\Rg (e_1e_2v)\varphi=(e_3v)\varphi.$$
Multiplying the relation $(ue_1)\varphi\Lg e_1\varphi$
through on the right by $e_2\varphi$, we obtain $(ue_3)\varphi=(ue_1e_2)\varphi\Lg (e_1e_2)\varphi=e_3\varphi$.
On the other hand, from the relation
\begin{equation}
\label{2}
e_1\varphi\Rg (e_3v)\varphi
\end{equation}
we have $e_1\varphi=(e_3v)\varphi\cdot s=e_3\varphi\cdot(v\varphi\cdot s)$
for some $s\in IG(E)$, and since $e_3\varphi=(e_1e_2)\varphi=e_1\varphi\cdot e_2\varphi$,
we conclude that $e_3\varphi\Rg e_1\varphi$. From this and from~\eqref{2} we get
$e_3\varphi\Rg (e_3v)\varphi$. Thus,
$$(ue_3)\varphi\Lg e_3\varphi\Rg (e_3v)\varphi,$$
as required. The situation when $x=u'e_1$, $e_2=f'$ and $y=v'$ is handled in a symmetric way.
\end{proof}

We are ready to state and to prove our main result.
\begin{theorem}
Let $\langle E,\circ\rangle$ be a biordered set, $IG(E)$ the free idempotent generated
semigroup on $E$. Every periodic element of $IG(E)$ lies in a subgroup of $IG(E)$.
\end{theorem}

\begin{proof}
Let $w=e_1\cdots e_n$, where $e_1,\dots,e_n\in E$, be a word in $E^+$ such that
$w\varphi\in IF(E)$ is periodic. Then $(w\varphi)^k=w^k\varphi$ is idempotent for
some $k$, whence, by Lemma~2, $w^k\varphi=e\varphi$ for some $e\in E$. If $k=1$,
there is nothing to prove, so we suppose $k>1$ and apply Lemma~3 to $w^k$.
It yields a decomposition of the form
$$w^k=(e_1\cdots e_n)^\ell e_1\cdots e_{i-1}\cdot e_i\cdot e_{i+1}\cdots e_n
(e_1\cdots e_n)^m$$
such that $0\le\ell,m<k$, $1\le i\le n$, and
\begin{equation}
\label{1}
\bigl((e_1\cdots e_n)^\ell e_1\cdots e_{i-1}e_i\bigr)\varphi\Lg e_i\varphi\Rg\bigl(e_ie_{i+1}\cdots e_n
(e_1\cdots e_n)^m\bigr)\varphi.
\end{equation}
Using Green's lemma, we deduce from~\eqref{1} the following relations:
$$\begin{array}{ccccc}
w^{\ell+1}\varphi & \stackrel{\Rg}{\rule{.8cm}{.4pt}} & (w^\ell e_1\cdots e_{i-1}e_i)\varphi
& \stackrel{\Rg}{\rule{.8cm}{.4pt}} & w^k\varphi=e\varphi\\
\rule{.4pt}{.8cm}\,\raisebox{.4cm}{\tiny$\Lg$} && \rule{1pt}{.8cm}\,\raisebox{.4cm}{\tiny$\Lg$}  &&\rule{.4pt}{.8cm}\,\raisebox{.4cm}{\tiny$\Lg$}  \\
(e_ie_{i+1}\cdots e_n)\varphi  &\stackrel{\Rg}{\rule{.8cm}{.4pt}} & e_i\varphi &\stackrel{\Rg}{\rule{.8cm}{1pt}}
& (e_ie_{i+1}\cdots e_n w^m)\varphi \\
\rule{.4pt}{.8cm}\,\raisebox{.4cm}{\tiny$\Lg$} && \rule{.4pt}{.8cm}\,\raisebox{.4cm}{\tiny$\Lg$}  &&\rule{.4pt}{.8cm}\,\raisebox{.4cm}{\tiny$\Lg$}  \\
w\varphi &\stackrel{\Rg}{\rule{.8cm}{.4pt}} & (e_1\cdots e_{i-1}e_i)\varphi &\stackrel{\Rg}{\rule{.8cm}{.4pt}} & w^{m+1}\varphi
\end{array}$$
(The ``initial'' relations in~\eqref{1} are represented by the bold lines.)
In particular, $w^{\ell+1}\varphi \Rg w^k\varphi$. Since $\ell+1\le k$,
we can apply Lemma~1 with $a=w\varphi$, $p=\ell+1$ and $q=k$, thus obtaining $w^{\ell+1}\varphi\Hg e\varphi$.
Hence $w\varphi\Lg e\varphi$ and the dual of Lemma~1 implies that $w\varphi\Hg e\varphi$,
that is, $w\varphi$ belongs to a subgroup of $IG(E)$.
\end{proof}

\textbf{Acknowledgement and a historical comment.} The authors are very much indebted to Stuart
Margolis for informing them about results of~\cite{BMM1,BMM2} and stimulating discussions. In
fact, the initial idea of this note arose in 1990 during the second and the third authors' visit
to the University of Sydney. However,  at that time it was not at all clear whether the free
idempotent generated semigroup on a biordered set may have non-idempotent periodic elements, and
publishing a result about objects that might not exist did not seem to be justified. It was not
until very recently that the examples in~\cite{BMM2} have confirmed that our theorem has indeed
a non-void applicability range.

\end{document}